\documentclass[oneside,english]{amsart}
\usepackage[T1]{fontenc}
\usepackage[latin9]{inputenc}
\usepackage{amsthm}
\usepackage{amstext}
\usepackage{amssymb}
\usepackage{esint}
\usepackage{color,enumitem,graphicx}

\makeatletter

\numberwithin{equation}{section}
\numberwithin{figure}{section}
\theoremstyle{plain}
\newtheorem{thm}{\protect\theoremname}[section]
  \theoremstyle{plain}
  
  \theoremstyle{plain}
  
  \theoremstyle{plain}
  
  \theoremstyle{plain}
  \newtheorem{lem}[thm]{\protect\lemmaname}
	\theoremstyle{plain}
  \newtheorem{rem}[thm]{\protect\remarkname}
  \theoremstyle{definition}
  \newtheorem{defn}[thm]{\protect\definitionname}

\makeatother

\usepackage{babel}
  \providecommand{\corollaryname}{Corollary}
  \providecommand{\definitionname}{Definition}
  \providecommand{\lemmaname}{Lemma}
	\providecommand{\remarkname}{Remark}
  \providecommand{\propositionname}{Proposition}
  \providecommand{\examplename}{Example}
  \providecommand{\theoremname}{Theorem}

\DeclareMathOperator{\dv}{div}

\DeclareMathOperator{\loc}{loc}

\DeclareMathOperator{\esssup}{esssup}

\begin{document}

\title[Singular subelliptic equations and Sobolev inequalities]{Singular subelliptic equations and Sobolev inequalities on nilpotent Lie groups}

\author{Prashanta Garain and Alexander Ukhlov}

\begin{abstract}
In this article we study singular subelliptic $p$-Laplace equations and best constants in Sobolev inequalities on nilpotent Lie groups. We prove solvability of these subelliptic $p$-Laplace equations and existence of the minimizer of the corresponding variational problem. It leads to existence of the best constant in the corresponding $(q,p)$-Sobolev inequality, $0<q<1$,  $1<p<\nu$.
\end{abstract}

\maketitle

\footnotetext{{\bf {Key words and phrases:}} Subelliptic operators, Carnot groups, singular problem, Sobolev inequality.}
\footnotetext{\textbf{2010 Mathematics Subject Classification:} {35H20, 22E30, 46E35}}

\section{Introduction }
In this article, we investigate the following singular Dirichlet boundary value problem for the subelliptic $p$-Laplace equation
\begin{equation}\label{subopsin}
-\dv_{\textrm{H}}(|\nabla_{\textrm{H}}u|^{p-2}\nabla_{\textrm{H}}u)=f(x)u^{-\delta}\text{ in }\Omega,\quad u>0\text{ in }\Omega,\quad u=0\text{ on }\partial\Omega,
\end{equation}
where $\Omega$ is a bounded domain on a stratified nilpotent Lie group $\mathbb G$ (Carnot group). 

We assume that $1<p<\nu$, $0<\delta<1$, where $\nu$ is the homogeneous dimension of $\mathbb G$ and $f\in L^m(\Omega)\setminus\{0\}$ is nonnegative, where $m\geq 1$ to be made precise below. More precisely, we establish existence, uniqueness and boundedness of weak solutions (Definition \ref{wssdef}) for the problem \eqref{subopsin}. Further, we observe that such solutions are associated to the following minimizing problem given by
\begin{equation}\label{subopmin}
\mu(\Omega):=\inf_{u\in W_{0}^{1,p}(\Omega)}\Bigg\{\int_{\Omega}|\nabla_{\textrm{H}}u|^p\,dx:\int_{\Omega}|u|^{1-\delta}f\,dx=1\Bigg\}.
\end{equation}

As a consequence, we obtain the following $(q,p)$-Sobolev inequality, $0<q=1-\delta<1$,  $1<p<\nu$:
\begin{equation}\label{subopSI}
S\left(\int_{\Omega}|u|^{1-\delta} f\,dx\right)^\frac{1}{1-\delta}\leq\left(\int_{\Omega}|\nabla_{\textrm{H}}u|^p\,dx\right)^\frac{1}{p},\quad\forall\,u\in W_0^{1,p}(\Omega),
\end{equation}
for some constant $S>0$.

By singularity, we mean the nonlinearity on the right hand side of \eqref{subopsin} blow up near the origin. Singular problems has been thoroughly studied over the last three decade and there is a colossal amount of literature in this concern. Most of these results are investigated in the Euclidean case and recently, some works has been done in the Riemannian manifolds as well.

To motivate our present study, let us state some known results related to our question. In the Euclidean case, Crandall-Rabinowitz-Tartar \cite{CRT} proved existence of a unique positive classical solution for the singular Laplace equation $-\Delta u=u^{-\delta}$, $\delta>0$ in a bounded smooth domain $\Omega\subset\mathbb{R}^N$ subject to the Dirichlet boundary condition $u=0$ on $\partial\Omega$. Lazer-McKenna \cite{LMpams} established that such solutions are weak solution in $W^{1,2}_{0}(\Omega)$ if and only if $0<\delta<3$. This restriction on $\delta$ has been removed by Boccardo-Orsina \cite{Boc} to investigate the existence of weak solutions in $W^{1,2}_{\mathrm{loc}}(\Omega)$. Indeed, they proved existence of weak solutions in $W_0^{1,2}(\Omega)$ when $0<\delta\leq 1$ and in $W^{1,2}_{\mathrm{loc}}(\Omega)$ such that $u^{\frac{\delta+1}{2}}\in W_0^{1,2}(\Omega)$ when $\delta>1$. We would like to point out that for $\delta>1$, the fact that $u^{\frac{\delta+1}{2}}\in W_0^{1,2}(\Omega)$ is referred to as the Dirichlet boundary condition $u=0$ on $\partial\Omega$. We also refer to \cite{arcoya, Merida} and the references therein in the semilinear context. Such results has been further extended to the singular $p$-Laplace equations by Canino-Sciunzi-Trombetta \cite{Canino}, De Cave \cite{DeCave}, Haitao \cite{YHaitao}, Giacomoni-Schindler-Tak\'{a}\v{c} \cite{GST}, Bal-Garain \cite{BGaniso} and the references therein.

As far as we are aware, singular problems are very less understood in the non-Euclidean setting. In the Riemannian setting, some regularity results for singular semilinear problems is studied in do \'O and Clemente \cite{doc}. Wang-Wang \cite{WW} obtained existence results for the purely singular problem and symmetry properties of solutions to the purturbed singular problem in the Heinsenberg group for the semilinear case. Recently, in the quasilinear case, Garain-Kinnunen \cite{GKpams} established non-existence result for the singular $p$-Laplace equation in a general metric measure space which satisfies a doubling property and a Poinc\'are inequality.

Our main purpose in this article is to investigate singular problems in the subelliptic setting on nilpotent Lie groups. To this end, as mentioned at the beginning, we obtain existence, uniqueness and boundedness of weak solutions (see Theorem \ref{subopsinex}) for the problem \eqref{subopsin}. Moreover, we prove that such solutions are associated to the best constant $\mu(\Omega)$ defined in \eqref{subopmin} which also gives rise to the Sobolev type inequality \eqref{subopSI} (see Theorem \ref{subopbc}). In the Euclidean case, such minimizing problems has been studied in Anello-Faraci-Iannizzotto \cite{GFA} for the local case, in the nonlocal case in Ercole-Pereira \cite{EP1}, for the weighted and anisotropic case in Bal-Garain \cite{BG}. These are based on the approximation technique introduced in Boccardo-Orsina \cite{Boc}. We adopt this technique in the subelliptic setting to prove our main results. To the best of our knowledge, our main results are new even for the linear case $p=2$. 

This article is organized as follows: In Section 2, we mention some preliminary results in the subelliptic setting. In Section 3 and 4, we state our main results and obtain some preliminary results respectively. Finally, in Section 5, we prove our main results.

\textbf{Notation:} We write $c$ or $C$ to denote a constant which may vary from line to line or even in the same line. If $c$ or $C$ depends on the parameters $r_1,r_2,\ldots,r_k$, we write $c=c(r_1,r_2,\ldots,r_k)$ or $C=C(r_1,r_2,\ldots,r_k)$ respectively. For $a\in\mathbb{R}$, we denote by $a^+=\max\{a,0\}$, $a^-=\min\{-a,0\}$.  For given constants $c,d$, a set $S$ and a function $u$, by $c\leq u\leq d$ in $S$, we mean $c\leq u\leq d$ almost everywhere in $S$.

\section{Nilpotent Lie groups and Sobolev spaces}

Recall that a stratified homogeneous group \cite{FS}, or, in another
terminology, a Carnot group \cite{Pa} is a~connected simply
connected nilpotent Lie group~$\mathbb G$ whose Lie algebra~$V$ is
decomposed into the direct sum~ $V_1\oplus\cdots\oplus V_m$ of
vector spaces such that $\dim V_1\geqslant 2$, $[V_1,\ V_i]=V_{i+1}$
for $1\leqslant i\leqslant m-1$ and $[V_1,\ V_m]=\{0\}$. Let
$X_{11},\dots,X_{1n_1}$ be left-invariant basis vector fields of
$V_1$. Since they generate $V$, for each $i$, $1<i\leqslant m$, one
can choose a basis $X_{ik}$ in $V_i$, $1\leqslant k\leqslant
n_i=\dim V_i$, consisting of commutators of order $i-1$ of fields
$X_{1k}\in V_1$. We identify elements $g$ of $\mathbb G$ with
vectors $x\in\mathbb R^N$, $N=\sum_{i=1}^{m}n_i$, $x=(x_{ik})$,
$1\leqslant i\leqslant m$, $1\leqslant k\leqslant n_i$ by means of
exponential map $\exp(\sum x_{ik}X_{ik})=g$. Dilations $\delta_t$
defined by the formula 
\begin{multline}
\nonumber
\delta_t x= (t^ix_{ik})_{1\leqslant i\leqslant m,\,1\leqslant k\leqslant n_j}\\
=(tx_{11},...,tx_{1n_1},t^2x_{21},...,t^2x_{2n_2},...,t^mx_{m1},...,t^mx_{mn_m}),
\end{multline}
are automorphisms of
$\mathbb G$ for each $t>0$. Lebesgue measure $dx$ on $\mathbb R^N$
is the bi-invariant Haar measure on~ $\mathbb G$ (which is generated
by the Lebesgue measure by means of the exponential map), and
$d(\delta_t x)=t^{\nu}~dx$, where the number
$\nu=\sum_{i=1}^{m}in_i$ is called the homogeneous dimension of the
group~$\mathbb G$. The measure $|E|$ of a measurable subset
$E$ of $\mathbb G$ is defined by
$
|E|=\int_E~dx.
$

Recall that a continuous map $\gamma: [a,b]\to\mathbb G$ is called a continuous curve on $\mathbb G$. This continuous curve is rectifiable if
$$
\sup\left\{\sum\limits_{k=1}^m|\left(\gamma(t_{k})\right)^{-1}\gamma(t_{k+1})|\right\}<\infty,
$$
where the supremum is taken over all partitions $a=t_1<t_2<...<t_m=b$ of the segment $[a,b]$.
In \cite{Pa} it was proved that any rectifiable curve is differentiable almost everywhere and $\dot{\gamma}(t)\in V_1$: there exists measurable functions $a_i(t)$, $t\in (a,b)$ such that
$$
\dot{\gamma}(t)=\sum\limits_{i=1}^n a_i(t)X_i(\gamma(t))\,\,\text{and}\,\,
\left|\left(\gamma(t+\tau)\right)^{-1}\gamma(t)exp(\dot{\gamma}(t)\tau)\right|=o(\tau)\,\,\text{as}\,\,\tau\to 0
$$
for almost all $t\in (a,b)$.
The length $l(\gamma)$ of a rectifiable curve $\gamma:[a,b]\to\mathbb G$ can be calculated by the formula
$$
l(\gamma)=\int\limits_a^b {\left\langle \dot{\gamma}(t),\dot{\gamma}(t)\right\rangle}_0^{\frac{1}{2}}~dt=
\int\limits_a^b \left(\sum\limits_{i=1}^{n}|a_i(t)|^2\right)^{\frac{1}{2}}~dt,
$$
where ${\left\langle \cdot,\cdot\right\rangle}_0$ is the inner product on $V_1$. The result of \cite{CH} implies that one can connect two arbitrary points $x,y\in \mathbb G$ by a rectifiable curve. The Carnot-Carath\'eodory distance $d(x,y)$ is the infimum of the lengths over all rectifiable curves with endpoints $x$ and $y$ in $\mathbb G$. The Hausdorff dimension of the metric space $\left(\mathbb G,d\right)$ coincides with the homogeneous dimension $\nu$ of the group $\mathbb G$.

\subsection{Sobolev spaces on Carnot groups}

Let $\mathbb G$ be a Carnot group with one-parameter dilatation
group $\delta_t$, $t>0$, and a homogeneous norm $\rho$, and let
$E$ be a measurable subset of $\mathbb G$. The Lebesgue space $L^p(E)$, $p\in [1,\infty]$, is the space of pth-power
integrable functions $f:E\to\mathbb R$ with the standard norm:
$$
\|f\|_{L^p(E)}=\biggl(\int\limits_{E}|f(x)|^p~dx\biggr)^{\frac{1}{p}},\,\,1\leq p<\infty,
$$
and $\|f\|_{L^{\infty}(E)}=\esssup_{E}|f(x)|$ for $p=\infty$. We
denote by $L^p_{\loc}(E)$ the space of functions
$f: E\to \mathbb R$ such that $f\in L^p(F)$ for each compact
subset $F$ of $E$.

Let $\Omega$ be an open set in $\mathbb G$. The (horizontal) Sobolev space
$W^{1,p}(\Omega)$, $1\leqslant p\leqslant\infty$, consists of the functions
$f:\Omega\to\mathbb R$ which are locally integrable in $\Omega$, having the weak
derivatives $X_{1i} f$ along the horizontal vector fields $X_{1i}$, $i=1,\dots,n_1$,
and the finite norm
$$
\|f\|_{W^{1,p}(\Omega)}=\|f\|_{L^p(\Omega)}+\|\nabla_{\textrm{H}} f\|_{L^p(\Omega)},
$$
where $\nabla_\textrm{H} f=(X_{11}f,\dots,X_{1n_1}f)$ is the horizontal subgradient of $f$.
If $f\in W^{1,p}(U)$ for each bounded open set $U$ such that
$\overline{U}\subset\Omega$ then we say that $f$ belongs to the
class $W^{1,p}_{\loc}(\Omega)$.

For the rest of the article, we assume that $\Omega\subset\mathbb{G}$ is a bounded domain. The Sobolev space $W^{1,p}_0(\Omega)$ is defined to be the closure of $C^{\infty}_c(\Omega)$ under the norm
$$
\|f\|_{W^{1,p}_0(\Omega)}=\|f\|_{L^p(\Omega)} +\|\nabla_{\textrm{H}} f\|_{L^p(\Omega)}.
$$
For the following result, refer to \cite{F75,Vodop, Vodop1, X90}. 
\begin{lem}
\label{Xuthm}
The space $W_0^{1,p}(\Omega)$ is a real  separable and uniformly convex Banach space.
\end{lem}

The following embedding result follows from \cite[$(2.8)$]{Danielli} and \cite{F75}, \cite[Theorem $8.1$]{HK00}, see also \cite[Theorem $2.3$]{Garo1}. 

\begin{lem}\label{emb}
Let $1<p<\nu$, then the space $W_0^{1,p}(\Omega)$ is continuously embedded in $L^q(\Omega)$ for every $1\leq q\leq p^{*}$ where $p^{*}={\nu p}/{(\nu-p)}$. Moreover, the embedding is compact for every $1\leq q<p^{*}$.
\end{lem}

Hence, for $1<p<\nu$ we can consider the Sobolev space $W_0^{1,p}(\Omega)$ with the norm
$$
\|f\|:=\|f\|_{W^{1,p}_0(\Omega)}=\|\nabla_{\textrm{H}} f\|_{L^p(\Omega)}.
$$

The following algebraic inequality from \cite[Lemma $2.1$]{Dama} will be useful for us.

\begin{lem}\label{alg}
Let $1<p<\infty$. Then for any $a,b\in\mathbb{R}^N$, there exists a positive constant $C=C(p)$ such that
\begin{equation}\label{algineq}
\langle |a|^{p-2}a-|b|^{p-2}b, a-b \rangle\geq
C(|a|+|b|)^{p-2}|a-b|^2.
\end{equation}
\end{lem}

\section{Singular subelliptic $p$-Laplace equations}

The system of basis vectors $X_1,X_2,\dots,X_n$ of the space $V_1$ (here and throughout we set $n_1=n$ and $X_{1i}=X_i$, where
$i=1,\dots,n$) satisfies the H\"ormander's hypoellipticity condition. We study the singular problem in the geometry of the vector fields  satisfying the H\"ormander's hypoellipticity condition.

\begin{defn}\label{wssdef}(Weak solution)
We say that $u\in W_0^{1,p}(\Omega)$ is a weak solution of \eqref{subopsin} if $u>0$ in $\Omega$ and for every $\omega\Subset\Omega$ there exists a positive constant $c(\omega)>0$ such that $u\geq c(\omega)>0$ in $\omega$ and for every $\phi\in C_c^{1}(\Omega)$, we have
\begin{equation}\label{wss}
\int_{\Omega}|\nabla_{\textrm{H}} u|^{p-2}\nabla_{\textrm{H}}u \nabla_{\textrm{H}}\phi\,dx=\int_{\Omega}f(x)u^{-\delta}\phi\,dx.
\end{equation}
\end{defn}

\begin{rem}\label{tstrmk}
First, we claim that if $u\in W_0^{1,p}(\Omega)$ is a weak solution of \eqref{subopsin}, then \eqref{wss} holds for every $\phi\in W_0^{1,p}(\Omega)$. 
Let $\psi\in W_0^{1,p}(\Omega)$, then there exists a sequence of nonnegative functions $\{\psi_n\}_{n\in\mathbb N}\subset C_c^{1}(\Omega)$ such that $0\leq \psi_n\to|\psi|$ strongly in $W_0^{1,p}(\Omega)$ as $n\to\infty$ and pointwise almost everywhere in $\Omega$. We observe that
\begin{equation}\label{unique}
\begin{split}
\Big|\int_{\Omega}f(x)u^{-\delta}\psi\,dx\Big|&\leq \int_{\Omega}f(x)u^{-\delta}|\psi|\,dx\leq\liminf_{n\to\infty}\int_{\Omega}f(x)u^{-\delta}\psi_n\,dx\\
&=\liminf_{n\to\infty}\langle -\dv_{\textrm{H}}(|\nabla_{\textrm{H}}u|^{p-2}\nabla_{\textrm{H}}u),\psi_n\rangle\\
&\leq \|u\|^{p-1}\lim_{n\to\infty}\|\psi_n\|\leq \|u\|^{p-1}\|\psi\|.
\end{split}
\end{equation}
Let $\phi\in W_0^{1,p}(\Omega)$, then there exists a sequence $\{\phi_n\}_{n\in\mathbb N}\subset C_c^{1}(\Omega)$, which converges to $\phi$ strongly in $W_0^{1,p}(\Omega)$. Choosing $\psi=\phi_n-\phi$ in \eqref{unique}, we obtain
\begin{equation}\label{lhslim}
\lim_{n\to\infty}\Big|\int_{\Omega}f(x)u^{-\delta}(\phi_n-\phi)\,dx\Big|\leq \|u\|^{p-1}\lim_{n\to\infty}\|\phi_n-\phi\|=0.
\end{equation}
Again, since $\phi_n\to\phi$ strongly in $W_0^{1,p}(\Omega)$ as $n\to\infty$, we have
\begin{equation}\label{rhslim}
\lim_{n\to\infty}\int_{\Omega}|\nabla_{\textrm{H}}~u|^{p-2}\nabla_{\textrm{H}}~u \nabla_{\textrm{H}}~(\phi_n-\phi)\,dx=0.
\end{equation}
Hence, using \eqref{lhslim} and \eqref{rhslim} in \eqref{wss} the claim follows.
\end{rem}

\textbf{Statement of the main results:} Our main results in this article are stated below. The first one is the following existence, uniqueness and regularity result.
\begin{thm}\label{subopsinex}
Let $0<\delta<1<p<\nu$ and $f\in L^m(\Omega)\setminus\{0\}$ be nonnegative, where $m=(\frac{p^*}{1-\delta})^{'}$. Then the problem \eqref{subopsin} admits a unique positive weak solution $u_\delta$ in $W_0^{1,p}(\Omega)$. Moreover, if $m>\frac{p^*}{p^* -p}$, then $u_{\delta}\in L^\infty(\Omega)$.
\end{thm}

The second main result asserts that the Sobolev inequality \eqref{subopSI} holds and its best constant $\mu(\Omega)$ defined above in \eqref{subopmin} is associated to the problem \eqref{subopsin}.
\begin{thm}\label{subopbc}
Let $0<\delta<1<p<\nu$ and $f\in L^m(\Omega)\setminus\{0\}$ be nonnegative, where $m=(\frac{p^*}{1-\delta})^{'}$. Assume that $u_\delta\in W_0^{1,p}(\Omega)$ is given by Theorem \ref{subopsinex}. Then, we have
\begin{equation}\label{bc}
\begin{split}
\mu(\Omega)&:=\inf_{u\in W_0^{1,p}(\Omega)}\Bigg\{\int_{\Omega}|\nabla_{\textrm{H}}u|^p\,dx:\int_{\Omega}|u|^{1-\delta}f\,dx=1\Bigg\}=\left(\int_{\Omega}|\nabla_{\textrm{H}}u_{\delta}|^p\,dx\right)^\frac{1-\delta-p}{1-\delta}.
\end{split}
\end{equation}
Further, the following Sobolev inequality
\begin{equation}\label{SI}
S\Big(\int_{\Omega}|v|^{1-\delta}f\,dx\Big)^\frac{p}{1-\delta}\leq\int_{\Omega}|\nabla_{\textrm{H}}v|^p\,dx,
\end{equation}
holds for every $v\in W_0^{1,p}(\Omega)$, if and only if $S\leq \mu(\Omega)$.
\end{thm}

\section{Preliminary results}
For $n\in\mathbb{N}$, we investigate the following approximated problem
\begin{equation}\label{mainapprox}
\begin{split}
-\dv_{\textrm{H}}(|\nabla_{\textrm{H}}~u|^{p-2}\nabla_{\textrm{H}}~u)=\frac{f_{n}}{(u^{+}+\frac{1}{n})^{\delta}}\text{ in }\Omega,\,u=0\text{ on }\partial\Omega,
\end{split}
\end{equation}
where $f_{n}(x)=\min\{f(x),n\}$, for $f\in L^{m}(\Omega)\setminus\{0\}$ is nonnegative, where $m=(\frac{p^*}{1-\delta})^{'}$, provided $0<\delta<1<p<\nu$ and $p^*=\frac{\nu p}{\nu-p}$.

\begin{lem}\label{auxresult}
Let $1<p<\nu$ and $g\in L^{\infty}(\Omega)\setminus\{0\}$ be nonnegative in $\Omega$. Then there exists a unique solution $u\in W_0^{1,p}(\Omega)$ of the problem
\begin{equation}\label{auxresulteqn}
\begin{split}
-\dv_{\textrm{H}}(|\nabla_{\textrm{H}}~u|^{p-2}\nabla_{\textrm{H}}~u)=g\text{ in }\Omega,\,u>0\text{ in }\Omega.
\end{split}
\end{equation}
Moreover, for every $\omega\Subset\Omega$, there exists a constant $c(\omega)$, satisfying $u\geq c(\omega)>0$ in $\omega$.
\end{lem}

\begin{proof}
\textbf{Existence:} We define the energy functional $I:W_0^{1,p}(\Omega)\to\mathbb{R}$ by
$$
I(u):=\frac{1}{p}\int_{\Omega}|\nabla_{\textrm{H}}u|^p\,dx-\int_{\Omega}g u\,dx.
$$
By using that $g\in L^\infty(\Omega)$ and Lemma \ref{emb}, we have
\begin{equation}\label{coercive}
\begin{split}
I(u)&\geq\frac{\|u\|^p}{p}-C|\Omega|^\frac{p-1}{p}\|g\|_{L^\infty(\Omega)}\|u\|,
\end{split}
\end{equation}
where $C>0$ is the Sobolev constant. Therefore, since $p>1$, we have $I$ is coercive. Moreover, $I$ is weakly lower semicontinuous. Indeed, we define $I=I_1+I_2$, where $I_1:W_0^{1,p}(\Omega)\to\mathbb{R}$ defined by
$$
I_1(u):=\frac{1}{p}\int_{\Omega}|\nabla_{\textrm{H}}u|^p\,dx
$$
and $I_2:W_0^{1,p}(\Omega)\to\mathbb{R}$ defined by
$$
I_2(u):=-\int_{\Omega}g u\,dx.
$$
Note that $I_1, I_2$ are convex and so $I$ is convex. Also, $I$ is a $C^1$ functional. Therefore, $I$ is weakly lower semicontinous. As a consequence $I$ has a minimizer, say $u\in W_0^{1,p}(\Omega)$ which solves the equation
\begin{equation}\label{auxeqn1}
-\dv_{\textrm{H}}(|\nabla_{\textrm{H}}~u|^{p-2}\nabla_{\textrm{H}}~u)=g\text{ in }\Omega.
\end{equation}
\textbf{Positivity:} Choosing $u_-:=\min\{u,0\}$ as a test function in \eqref{auxeqn1} and since $g\geq 0$, we obtain
\begin{equation*}
\begin{split}
\int_{\Omega}|\nabla_{\textrm{H}}u_-|^p\,dx=\int_{\Omega}g u_-\,dx\leq 0,
\end{split}
\end{equation*}
which gives, $u\geq 0$ in $\Omega$. Further $g\neq 0$ gives $u\neq 0$ in $\Omega$. Applying \cite[Theorem $5$]{Vodop} for every $\omega\Subset\Omega$, there exists a constant $c(\omega)$ such that $u\geq c(\omega)>0$ in $\Omega$. Thus $u>0$ in $\Omega$.\\
\textbf{Uniqueness:} Let $u,v\in W_0^{1,p}(\Omega)$ solves the problem \eqref{auxeqn1}. Therefore,
\begin{equation}\label{unif1}
\int_{\Omega}|\nabla_{\textrm{H}}u|^{p-2}\nabla_{\textrm{H}}u \nabla_{\textrm{H}}\phi\,dx=\int_{\Omega}g\phi\,dx,
\end{equation}
and 
\begin{equation}\label{unif22}
\int_{\Omega}|\nabla_{\textrm{H}}v|^{p-2}\nabla_{\textrm{H}}v \nabla_{\textrm{H}}u\phi\,dx=\int_{\Omega}g\phi\,dx
\end{equation}
holds for every $\phi\in W_0^{1,p}(\Omega)$. We choose $\phi=(u-v)^+$ and then subtracting \eqref{unif1} with \eqref{unif22}, we obtain
\begin{equation}\label{unif3}
\int_{\Omega}\langle|\nabla_{\textrm{H}}u|^{p-2}\nabla_{\textrm{H}}u-|\nabla_{\textrm{H}}v|^{p-2}\nabla_{\textrm{H}}v,\nabla_{\textrm{H}}(u-v)^+\rangle\,dx=0.
\end{equation}
By Lemma \ref{alg}, we get $u\leq v$ in $\Omega$. Similarly, we get $v\leq u$ in $\Omega$. Hence the uniqueness follows.
\end{proof}

\begin{lem}\label{exisapprox}
For every $n\in\mathbb{N}$, the problem \eqref{mainapprox} admits a solution $u_{n}\in W_0^{1,p}(\Omega)$ which is positive in $\Omega$. Moreover, $u_n$ is unique and $u_{n+1}\geq u_{n}$ in $\Omega$ for every $n$, and for every $\omega\Subset\Omega$, there exists a constant $c(\omega)$ (independent of $n$) such that $u_{n}\geq c(\omega)>0$ in $\omega$. Further, $\|u_n\|\leq c$, for some positive constant $c$, which is independent of $n$.
\end{lem}

\begin{proof}
\textbf{Existence:} Let $n\in\mathbb N$, then for every $h\in L^p(\Omega)$, there exists a unique $u\in W_0^{1,p}(\Omega)$ such that
\begin{equation}\label{approxfixed}
-\dv_{\textrm{H}}(|\nabla_{\textrm{H}}u|^{p-2}\nabla_{\textrm{H}}u)=\frac{f_{n}}{(h^{+}+\frac{1}{n})^\delta}\text{ in }\Omega.
\end{equation}
Therefore, we define $T:L^p(\Omega)\to L^p(\Omega)$ by $T(h)=I(u)=u$, where $u$ solves \eqref{approxfixed} and $I:W_0^{1,p}(\Omega)\to L^p(\Omega)$ is the continuous and compact inclusion mapping from Lemma \ref{emb}. First we observe that $T$ is continuous. Indeed, let $\{h_k\}_{k\in\mathbb N}\subset L^p(\Omega)$ and $h\in L^p(\Omega)$ be such that $h_k\to h$ in $L^p(\Omega)$. Suppose $T(h_k)=u_k$ and $T(h)=u$. Then we claim that $u_k\to u$ in $L^p(\Omega)$. By the definition of the mapping $T$, for every $\phi\in W_0^{1,p}(\Omega)$, we have
\begin{equation}\label{T1}
\int_{\Omega}|\nabla_{\textrm{H}} u_k|^{p-2}\nabla_{\textrm{H}} u_k \nabla_{\textrm{H}}\phi\,dx=\int_{\Omega}\frac{f_n}{(h_{k}^+ +\frac{1}{n})^\delta}\phi\,dx
\end{equation}
and
\begin{equation}\label{T2}
\int_{\Omega}|\nabla_{\textrm{H}}u|^{p-2}\nabla_{\textrm{H}}u X\phi\,dx=\int_{\Omega}\frac{f_n}{(h^+ +\frac{1}{n})^\delta}\phi\,dx.
\end{equation}
Let
$$
h_{k,n}=\big((h_{k}^+ +\frac{1}{n})^{-\delta}-(h^+ +\frac{1}{n})^{-\delta}\big).
$$
Then, choosing $\phi=u_k-u$ in \eqref{T1} and \eqref{T2} and then subtracting the resulting equations and using Lemma \ref{alg}, we get
\begin{equation}\label{T3}
\begin{split}
&\int_{\Omega}\langle|\nabla_{\textrm{H}} u_k|^{p-2}\nabla_{\textrm{H}} u_k-|\nabla_{\textrm{H}}u|^{p-2}\nabla_{\textrm{H}}u,\nabla_{\textrm{H}}(u_k-u)\rangle\,dx\\
&=\int_{\Omega}f_n\big((h_{k}^+ +\frac{1}{n})^{-\delta}-(h^+ +\frac{1}{n})^{-\delta}\big)(u_k-u)\,dx
\leq n\int_{\Omega}|h_{k,n}||u_k-u|\,dx\\
&\leq n \|h_{k,n}\|_{L^{(p^*)'}(\Omega)}\|u_k-u\|_{L^{p^*}(\Omega)}\leq C n \|h_{k,n}\|_{L^{(p^*)'}(\Omega)}\|u_k-u\|,
\end{split}
\end{equation}
where $C>0$ is the Sobolev constant. Thus, using Lemma \ref{alg} in \eqref{T3} we have
\begin{equation}\label{T4}
\|u_k-u\|\leq Cn\|h_{k,n}\|^\frac{1}{t-1}_{L^{(p^*)'}(\Omega)},
\end{equation}
where $t=p$ if $p\geq 2$ and $t=2$ if $p<2$. Note that upto a subsequenc, $h_{k,n}\to 0$ as $k\to\infty$ pointwise almost everywhere in $\Omega$ and $|h_{k,n}|\leq 2n^{\delta+1}$ and thus, by the Lebsegue dominated convergence theorem, from \eqref{T4} upto a subsequence
$
\|u_k-u\|\to 0
$ as $k\to\infty$. Since the limit is independent of the choice of the subsequence, by Lemma \ref{emb}, $u_k\to u$ strongly in $L^p(\Omega)$. Therefore, the continuity of $T$ follows. 

To prove that $T$ is compact, choosing $u$ as a test function in \eqref{approxfixed} and by Lemma \ref{emb}, we get
\begin{align*}
\|u\|^p\leq\int_{\Omega}n^{\delta+1}u\,dx\leq n^{\delta+1}|\Omega|^\frac{p-1}{p}C\|u\|,
\end{align*}
where $C>0$ is the Sobolev constant. Thus, 
\begin{equation}\label{T5}
\|u\|\leq c,
\end{equation}
where $c>0$ is a constant independent of $h$. Let $\{h_k\}_{k\in\mathbb N}\subset L^p(\Omega)$ be a bounded sequence, then by \eqref{T5} we have
$$
\|T(h_k)\|\leq c,
$$
where $c>0$ is a constant, which is independent of $h_k$. Thus, by the compactness of the embedding $W_0^{1,p}(\Omega)\to L^p(\Omega)$ from Lemma \ref{emb} we deduce that the sequence $\{T(h_k)\}_{k\in\mathbb N}$ extract a subsequence which converges strongly in $L^p(\Omega)$. Thus, $T$ is compact.

Let us set
$$
S:=\{h\in L^p(\Omega):\lambda T(h)=h,\quad 0\leq\lambda\leq 1\}.
$$
Then for any $h_1,h_2\in S$, by \eqref{T5} we have
$$
\|h_1-h_2\|_{L^p(\Omega)}=\lambda\|T(h_1)-T(h_2)\|_{L^p(\Omega)}\leq c,
$$
for some constant $c>0$, which is independent of $h_1,h_2$. Thus, by Schauder's fixed point theorem, there exists a fixed point $u_n\in W_0^{1,p}(\Omega)$ such that $T(u_n)=u_n$. As a consequence, $u_n$ solves the problem \eqref{mainapprox}. Moreover, by Lemma \ref{auxresult}, we have $u_n>0$ in $\Omega$ and for every $\omega\Subset\Omega$, there exists a constant $c(\omega)$ such that 
\begin{equation}\label{pos}
u_1\geq c(\omega)>0\text{ in }\omega.
\end{equation}
\textbf{Monotonicity and uniqueness:} Choosing $\phi = (u_{n}-u_{n+1})^+$ as a test function in \eqref{mainapprox} we have
\begin{equation}\label{T6}
\begin{split}
J&=\langle|\nabla_{\textrm{H}}\,u_n|^{p-2}\nabla_{\textrm{H}}\,u_n-|\nabla_{\textrm{H}}\,u_{n+1}|^{p-2}\nabla_{\textrm{H}}\,u_{n+1},\nabla_{\textrm{H}}\,(u_{n}-u_{n+1})^{+}\rangle\\
&=\int_{\Omega}\Big\{\frac{f_{n}}{\big(u_{n}+\frac{1}{n}\big)^\delta}-\frac{f_{n+1}}{\big(u_{n+1}+\frac{1}{n+1}\big)^\delta}\Big\}(u_{n}-u_{n+1})^+\,dx.
\end{split}
\end{equation}
Using the inequalities $f_{n}(x) \leq f_{n+1}(x)$, we obtain
\begin{align*}
J&=\int_{\Omega}\Big\{\frac{f_{n}}{\big(u_{n}+\frac{1}{n}\big)^\delta}-\frac{f_{n+1}}{\big(u_{n+1}+\frac{1}{n+1}\big)^\delta}\Big\}(u_{n}-u_{n+1})^+\,dx\\
&\leq\int_{\Omega}f_{n+1}\Big\{\frac{1}{\big(u_{n}+\frac{1}{n}\big)^\delta}-\frac{1}{\big(u_{n+1}+\frac{1}{n+1}\big)^\delta}\Big\}(u_{n}-u_{n+1})^+\,dx\leq 0.
\end{align*}
Noting this fact and using Lemma \ref{alg} in \eqref{T6}, it follows that $u_{n+1}\geq u_{n}$ in $\Omega$. Uniqueness follows similarly. From \eqref{pos}, we know that $u_{1}\geq c(\omega)>0$ for every $\omega\Subset\Omega$. Hence using the monotonicity, for every $\omega\Subset\Omega$, we get $u_{n}\geq c(\omega)>0$ in $\omega$, for some positive constant $c(\omega)$ (independent of $n$).\\
\textbf{Uniform boundedness:} Choosing $u_{n}$ as a test function in \eqref{mainapprox} and by Lemma \ref{emb}, we get
\begin{equation}\label{unif}
\begin{split}
\|u_{n}\|^p\leq\int_{\Omega}f u_{n}^{1-\delta}\,dx&\leq \|f\|_{L^{m}(\Omega)}\Big(\int_{\Omega}u_{n}^{(1-\delta)m'}\,dx\Big)^\frac{1}{m'}=\|f\|_{L^{m}(\Omega)}\Big(\int_{\Omega}u_{n}^{p^{*}}\,dx\Big)^\frac{1-\delta}{p^{*}}\\
&\leq c\|f\|_{L^{m}(\Omega)}\|u_n\|^{1-\delta},
\end{split}
\end{equation}
for some constant $c>0$, independent of $n$. Therefore, we have $\|u_{n}\|\leq c$, for some positive constant $c$ (independent of $n$).
\end{proof}

\begin{rem}\label{ptrmk}
As a consequence of Lemma \ref{exisapprox} we define the pointwise limit $u_\delta$ of $u_n$ in $\Omega$. Note that by the monotonicity of $u_n$ in Lemma \ref{exisapprox} we have $u_\delta\geq u_n$ in $\Omega$ for every $n\in\mathbb N$.
\end{rem}

\begin{lem}\label{essbdd}
Suppose that $f\in L^q(\Omega)\setminus\{0\}$ is nonnegative, where $q>\frac{p^*}{p^{*}-p}$. Then $\|u_n\|_{L^\infty(\Omega)}\leq C$, for some positive constant $C$ independent of $n$, where $\{u_n\}_{n\in\mathbb N}$ is the sequence of solutions of the problem \eqref{mainapprox} given by Lemma \ref{exisapprox}.
\end{lem}

\begin{proof}
For $k\geq 1$, we define $A(k)=\{x\in\Omega:u_n(x)\geq k\}$. Choosing $\phi_k(x)=(u_n-k)^{+}$ as a test function in \eqref{mainapprox}, first using H\"older's inequality with the exponents $p^{*'}, p^*$ and then, by Young's inequality with exponents $p$ and $p'$, we obtain
\begin{equation}\label{unibdd}
\begin{split}
\|\phi_k\|^p=\int_{\Omega}\frac{f_n}{\big(u_n+\frac{1}{n}\big)^\delta}\phi_k\,dx&\leq\int_{A(k)}f(x)\phi_k\,dx\leq\left(\int_{A(k)}f^{p^{*'}}\,dx\right)^\frac{1}{p^{*'}}\left(\int_{\Omega}\phi_k^{p^*}\,dx\right)^\frac{1}{p^*}\\
&\leq C\left(\int_{A(k)}f^{p^*{'}}\,dx\right)^\frac{1}{p^{*'}}\|\phi_k\|\\
&\leq \epsilon\|\phi_k\|^p+C(\epsilon)\left(\int_{A(k)}f^{p^{*'}}\,dx\right)^\frac{p'}{p^{*'}},
\end{split}
\end{equation}
where $C$ is the Sobolev constant and $C(\epsilon)>0$ is some constant depending on $\epsilon\in(0,1)$ but independent of $n$. Note that $q>\frac{p^*}{p^* -p}$ gives $q>p^{*'}$. Therefore, fixing $\epsilon\in(0,1)$ and again using H\"older's inequality with exponents $\frac{q}{p^{*'}}$ and $\big(\frac{q}{p^{*'}}\big)'$, for some constant $C>0$ which is independent of $n$, we obtain
\begin{equation}\label{supbd1}
\begin{split}
\|\phi_k\|^p\leq C\left(\int_{A(k)}f^{p^{*'}}\,dx\right)^\frac{p'}{p^{*'}}&\leq C\left(\int_{\Omega}f^q\,dx\right)^\frac{p'}{q}|A(k)|^{\frac{p'}{p^{*'}}\frac{1}{\big(\frac{q}{p^{*'}}\big)'}}\\
&\leq C|A(k)|^{\frac{p'}{p^{*'}}\frac{1}{\big(\frac{q}{p^{*'}}\big)'}}.
\end{split}
\end{equation}
Let $h>0$ be such that $1\leq k<h$. Then, $A(h)\subset A(k)$ and for any $x\in A(h)$, we have $u_n(x)\geq h$. So, $u_n(x)-k\geq h-k$ in $A(h)$. Combining these facts along with \eqref{supbd1} and again using Lemma \ref{emb} for some constant $C>0$ (independent of $n$), we arrive at
\begin{align*}
(h-k)^p|A(h)|^\frac{p}{p^*}&\leq\left(\int_{A(h)}(u_n-k)^{p^*}\,dx\right)^\frac{p}{p^{*}}\leq\left(\int_{A(k)}(u_n-k)^{p^*}\,dx\right)^\frac{p}{p^{*}}\\
&\leq C\|\phi_k\|^p\leq C|A(k)|^{\frac{p'}{p^{*'}}\frac{1}{\big(\frac{q}{p^{*'}}\big)'}}.
\end{align*}
Thus, for some constant $C>0$ (independent of $n$), we have
$$
|A(h)|\leq \frac{C}{(h-k)^{p^*}}|A(k)|^{\alpha},
$$
where
$$
\alpha={\frac{p^{*}p'}{pp^{*'}}\frac{1}{\big(\frac{q}{p^{*'}}\big)'}}.
$$
Due to the assumption, $q>\frac{p^*}{p^* -p}$, we have $\alpha>1$. Hence, by \cite[Lemma B.1]{Stam}, we have
$
\|u_n\|_{L^\infty(\Omega)}\leq C,
$
for some positive constant $C>0$ independent of $n$.
\end{proof}

\begin{lem}\label{lemma1}
Let $\{u_n\}_{n\in\mathbb N}$ be the sequence of solutions of the problem \eqref{mainapprox} given by Lemma \ref{exisapprox}. Then, for every $n\in\mathbb N$, we have
\begin{equation}\label{prop1}
\|u_{n}\|^{p}\leq\|\phi\|^{p}+p\int_{\Omega}\frac{(u_{n}-\phi)}{(u_{n}+\frac{1}{n})^\delta}f_{n}\,dx,\quad \forall\phi\in W_{0}^{1,p}(\Omega).
\end{equation}
Moreover, for every $n\in\mathbb{N}$, we have
\begin{equation}\label{nm}
\|u_{n}\|\leq\|u_{n+1}\|.
\end{equation}
\end{lem}

\begin{proof}
By Lemma \ref{auxresult}, for any $h\in W_0^{1,p}(\Omega)$, there exists a unique solution $v\in W_0^{1,p}(\Omega,w)$ of the problem
\begin{equation}\label{auxeqn}
-\dv_{\textrm{H}}(|\nabla_{\textrm{H}}u|^{p-2}\nabla_{\textrm{H}}v)=\frac{f_{n}(x)}{(h^{+}+\frac{1}{n})^\delta},\quad v>0\text{ in }\Omega,\quad v=0\text{ on }\partial\Omega.
\end{equation}
Note that $v$ is also a minimizer of the functional $J:W_0^{1,p}(\Omega)\to\mathbb{R}$ defined by
$$
J(\phi):=\frac{1}{p}\|\phi\|^{p}-\int_{\Omega}\frac{f_{n}}{(h^{+}+\frac{1}{n})^\delta}\phi\,dx.
$$
Thus, $J(v)\leq J(\phi)$, for every $\phi\in W_0^{1,p}(\Omega)$ and we obtain
\begin{equation}\label{mineqn}
\frac{1}{p}\|v\|^{p}-\int_{\Omega}\frac{f_{n}}{(h^{+}+\frac{1}{n})^\delta}v\,dx\leq \frac{1}{p}\|\phi\|^{p}-\int_{\Omega}\frac{f_{n}}{(h^{+}+\frac{1}{n})^\delta}\phi\,dx.
\end{equation}
Setting $v=h=u_{n}$ (which is positive in $\Omega$) in \eqref{mineqn}, the estimate \eqref{prop1} follows.\\
By Lemma \ref{exisapprox}, we know $u_{n}\leq u_{n+1}$ for every $n\in\mathbb N$. Thus, choosing $\phi=u_{n+1}$ in \eqref{prop1} we obtain $\|u_{n}\|\leq \|u_{n+1}\|$. Hence the result follows.
\end{proof}

\begin{lem}\label{lemma2}
Let $\{u_n\}_{n\in\mathbb N}$ be the sequence of solutions of the problem \eqref{mainapprox} given by Lemma \ref{exisapprox} and suppose $u_\delta$ is the pointwise limit of $u_n$ given by Remark \ref{ptrmk}. Then, upto a subsequence 
\begin{equation}\label{sc}
u_n\to u_{\delta}\text{ strongly in }W_0^{1,p}(\Omega).
\end{equation}
Moreover, $u_{\delta}$ is a minimizer of the energy functional $J_{\delta}:W_0^{1,p}(\Omega)\to\mathbb{R}$ defined by
\begin{align}\label{eng}
J_{\delta}(v):=\frac{1}{p}\|v\|^p-\frac{1}{1-\delta}\int_{\Omega}(v^{+})^{1-\delta}f\,dx.
\end{align}
\end{lem}
\begin{proof}
By Remark \ref{ptrmk}, we know that $u_n\leq u_\delta$ in $\Omega$ for every $n\in\mathbb N$. Thus, by choosing $\phi=u_{\delta}$ in \eqref{prop1} we get $\|u_{n}\|\leq \|u_{\delta}\|$ for every $n\in\mathbb N$. Hence, using the norm monotonicity $\|u_n\|\leq \|u_{n+1}\|$ from \eqref{nm}, we have
\begin{equation}\label{lim1}
\lim_{n\to\infty}\|u_{n}\|\leq \|u_{\delta}\|.
\end{equation}
Moreover, since $u_{n}\rightharpoonup u_{\delta}$ weakly in $W_0^{1,p}(\Omega)$, we get
\begin{equation}\label{lim2}
\|u_{\delta}\|\leq \liminf_{n\to\infty}\|u_{n}\|.
\end{equation}
Thus from \eqref{lim1} and \eqref{lim2} along with the uniform convexity of $W_0^{1,p}(\Omega)$ from Lemma \ref{Xuthm}, the convergence in \eqref{sc} follows.

To prove the second result, it is enough to show that, for all $v\in W_0^{1,p}(\Omega)$, we have
\begin{equation}\label{limclm}
J_{\delta}(u_{\delta})\leq J_{\delta}(v).
\end{equation}
To this end, for any $n\in\mathbb N$ and recalling that $f_n(x)=\min\{f(x),n\}$, we define $J_{n}:W_0^{1,p}(\Omega)\to\mathbb{R}$ by
$$
J_{n}(v):=\frac{1}{p}\|v\|^{p}-\int_{\Omega}H_n(v)f_{n}\,dx,
$$
where
$$
H_n(t):=\frac{1}{1-\delta}\Big(t^{+}+\frac{1}{n}\Big)^{1-\delta}-\Big(\frac{1}{n}\Big)^{-\delta}t^{-}.
$$
Then we observe that $J_{n}$ is $C^1$, bounded below and coercive, and therefore, $J_{n}$ has a minimizer, say $v_{n}\in W_0^{1,p}(\Omega)$. This gives, $J_n(v_n)\leq J_n(v_n^{+})$, from which it follows that $v_{n}\geq 0$ in $\Omega$. Noting that $\langle J_{n}{'}(v_{n}),\phi\rangle=0$ for all $\phi\in W_0^{1,p}(\Omega)$, we conclude that $v_{n}$ solves \eqref{mainapprox}. By the uniqueness result from Lemma \ref{exisapprox}, we get $u_{n}=v_{n}$ in $\Omega$. Hence $u_{n}$ is a minimizer of $J_{n}$ and we have
\begin{equation}\label{Inlim}
J_n(u_n)\leq J_n(v^+),\quad\forall\,v\in W_0^{1,p}(\Omega).
\end{equation}
Below we pass to the limit as $n\to\infty$ in \eqref{Inlim} to prove the claim \eqref{limclm}. Indeed, by Remark \ref{ptrmk} we know that $u_{n}\leq u_{\delta}$ in $\Omega$, which along with the Lebesgue dominated convergence theorem gives us
\begin{equation}\label{Inlim1}
\begin{split}
\lim_{n\to\infty}\int_{\Omega}H_n(u_{n})f_{n}\,dx&=\frac{1}{1-\delta}\int_{\Omega}(u_{\delta})^{1-\delta}f\,dx.
\end{split}
\end{equation}
Also, by the strong convergence from \eqref{sc}, we have 
\begin{equation}\label{Inlim2}
\lim_{n\to\infty}\|u_{n}\|= \|u_{\delta}\|.
\end{equation}
Hence, by \eqref{Inlim1} and \eqref{Inlim2}, we have
\begin{equation}\label{Inlim3}
\lim_{n\to\infty}J_{n}(u_{n})=J_{\delta}(u_{\delta}).
\end{equation}
Again,
\begin{equation}\label{Jnlim1}
\begin{split}
\lim_{n\to\infty}\int_{\Omega}H_n(v^{+})f_{n}\,dx&=\frac{1}{1-\delta}\int_{\Omega}(v^{+})^{1-\delta}f\,dx,
\end{split}
\end{equation}
for any $v\in W_0^{1,p}(\Omega)$. Now, letting $n\to\infty$ in \eqref{Inlim} and then using the estimates \eqref{Inlim3}, \eqref{Jnlim1} along with $\|v^{+}\|\leq \|v\|$, the inequality \eqref{limclm} holds. Hence the result follows.
\end{proof}

\section{Proof of the main results}
\textbf{Proof of Theorem \ref{subopsinex}:}\\
\textbf{Uniqueness:} Suppose $u,v\in W_0^{1,p}(\Omega)$ are weak solutions of \eqref{subopsin}. Then by Remark \ref{tstrmk} choosing $\phi=(u-v)^{+}\in W_0^{1,p}(\Omega)$ as a test function in \eqref{wss} we have
\begin{equation}\label{uni1}
\int_{\Omega}|\nabla_{\textrm{H}}u|^{p-2}\nabla_{\textrm{H}}u \nabla_{\textrm{H}}(u-v)^+\,dx=\int_{\Omega}fu^{-\delta}(u-v)^+\,dx,
\end{equation}
\begin{equation}\label{uni2}
\int_{\Omega}|\nabla_{\textrm{H}}u|^{p-2}\nabla_{\textrm{H}}u \nabla_{\textrm{H}}(u-v)^+\,dx=\int_{\Omega}fv^{-\delta}(u-v)^+\,dx.
\end{equation}
Subtracting \eqref{uni1} and \eqref{uni2}, we have
\begin{align*}
\int_{\Omega}\langle|\nabla_{\textrm{H}}u|^{p-2}\nabla_{\textrm{H}}u-|\nabla_{\textrm{H}}v|^{p-2}\nabla_{\textrm{H}}v,\nabla_{\textrm{H}}(u-v)^+\rangle\,dx=\int_{\Omega}f(u^{-\delta}-v^{-\delta})(u-v)^{+}\,dx\leq 0.
\end{align*}
Therefore, by Lemma \ref{alg}, we obtain $u\leq v$ in $\Omega$. In a similar way, we have $v\leq u$ in $\Omega$. Hence the result follows.\\
\textbf{Existence:} For every $n\in\mathbb{N}$, using Lemma \ref{exisapprox}, there exists $u_{n}\in W_0^{1,p}(\Omega)$ such that
\begin{equation}\label{exislim1}
\int_{\Omega}|\nabla_{\textrm{H}} u_n|^{p-2}\nabla_{\textrm{H}}u_n \nabla_{\textrm{H}}\phi\,dx=\int_{\Omega}\frac{f_{n}}{(u_n+\frac{1}{n})^{\delta}}\phi\,dx,\quad\forall\,\phi\in C_c^{1}(\Omega).
\end{equation}
\textbf{Limit pass:} By the strong convergence \eqref{sc} in Lemma \ref{lemma1}, upto a subsequence, we have $\nabla_{\textrm{H}}u_{n}\to \nabla_{\textrm{H}}u_{\delta}$ pointwise almost everywhere in $\Omega$ as $n\to\infty$. Hence, for every $\phi\in C_c^{1}(\Omega)$, we have
\begin{equation}\label{exislim4}
\begin{split}
\lim_{n\to\infty}\int_{\Omega}|\nabla_{\textrm{H}} u_n|^{p-2}\nabla_{\textrm{H}} u_n \nabla_{\textrm{H}}\phi\,dx&=\int_{\Omega}|\nabla_{\textrm{H}} u_\delta|^{p-2}\nabla_{\textrm{H}} u_\delta \nabla_{\textrm{H}}\phi\,dx.
\end{split}
\end{equation}
Let $\mathrm{supp}\,\phi=\omega\Subset\Omega$ and so by Lemma \ref{exisapprox}, there exists a constant $c(\omega)>0$, which is independent of $n$ such that $u_n\geq c(\omega)>0$ in $\omega$. Thus, $u_\delta\geq c(\omega)>0$ in $\omega$ and also we have
\begin{equation*}\label{exislim2}
\frac{f_{n}}{u_{n}^{\delta}}\phi\leq\frac{f}{c(\omega)^{\delta}}\|\phi\|_{L^{\infty}(\Omega)}\in L^1(\Omega).
\end{equation*}
By Remark \ref{ptrmk}, using the pointwise convergence $u_n\to u_{\delta}$ almost everywhere in $\Omega$ and the Lebesgue dominated convergence theorem, we have
\begin{equation}\label{exislim3} \lim_{n\to\infty}\int_{\Omega}\frac{f_{n}}{(u_{n}+\frac{1}{n})^{\delta}}\phi\,dx=\int_{\Omega}\frac{f}{u_{\delta}^{\delta}}\phi\,dx.
\end{equation}
Combining the estimates \eqref{exislim4} and \eqref{exislim3} in \eqref{exislim1}, we obtain $u_\delta$ is a weak solution of \eqref{subopsin}.\\
\textbf{Boundedness:} Using Lemma \ref{essbdd} we have $u_\delta\in L^\infty(\Omega)$.
\qed\\
\textbf{Proof of Theorem \ref{subopbc}:} First, we prove \eqref{bc}. Let us set 
$$
S_{\delta}:=\Bigg\{v\in W_0^{1,p}(\Omega):\int_{\Omega}|v|^{1-\delta}f\,dx=1\Bigg\}.
$$
Thus it is enough to obtain
\begin{align*}
\mu(\Omega)&:=\inf_{v\in S_{\delta}}\|v\|^{p}=\|u_{\delta}\|^{\frac{p(1-\delta-p)}{1-\delta}}.
\end{align*}
We observe that $U_{\delta}=\theta_{\delta} u_{\delta}\in S_{\delta}$, where
$$
\theta_{\delta}=\Bigg(\int_{\Omega}u_{\delta}^{1-\delta}f\,dx\Bigg)^{-\frac{1}{1-\delta}}.
$$
By Remark \ref{tstrmk}, choosing $\phi=u_{\delta}\in W_0^{1,p}(\Omega)$ as a test function in $\eqref{wss}$, we have
\begin{equation}\label{useful1}
\begin{split}
\int_{\Omega}|\nabla_{\textrm{H}} u_\delta|^p\,dx=\|u_{\delta}\|^p=\int_{\Omega}u_{\delta}^{1-\delta}f\,dx.
\end{split}
\end{equation}
Therefore, we have
\begin{equation}\label{extremal}
\begin{split}
\|U_{\delta}\|^{p}&=\int_{\Omega}|\nabla_{\textrm{H}} U_\delta|^p\,dx=\theta_\delta^p\int_{\Omega}|\nabla_{\textrm{H}} u_\delta|^p\,dx=\|u_{\delta}\|^\frac{p(1-\delta-p)}{1-\delta}.
\end{split}
\end{equation}
Let $v\in S_{\delta}$ and define by $\mu=\|v\|^{-\frac{p}{p+\delta-1}}$. Then by Lemma \ref{lemma1}, since $u_{\delta}$ is a minimizer of the functional $J_{\delta}$ given by \eqref{eng}, we have 
\begin{equation}\label{m1}
J_{\delta}(u_\delta)\leq J_{\delta}(\mu|v|).
\end{equation}
Using \eqref{useful1}, we have
\begin{equation}\label{m2}
\begin{split}
J_{\delta}(u_\delta)=\frac{1}{p}\|u_\delta\|^p-\frac{1}{1-\delta}\int_{\Omega}u_\delta^{1-\delta}f\,dx=\Big(\frac{1}{p}-\frac{1}{1-\delta}\Big)\|u_{\delta}\|^{p}.
\end{split}
\end{equation}
Again, since $v\in S_\delta$, we have
\begin{equation}\label{m3}
\begin{split}
J_{\delta}(\mu|v|)&=\frac{\mu^{p}}{p}\||v|\|^{p}-\frac{\mu^{1-\delta}}{1-\delta}\leq\frac{\mu^{p}}{p}\|v\|^{p}-\frac{\mu^{1-\delta}}{1-\delta}=\Big(\frac{1}{p}-\frac{1}{1-\delta}\Big)\|v\|^\frac{p(\delta-1)}{\delta+p-1}.
\end{split}
\end{equation}
Since $v\in S_{\delta}$ is arbitrary, using \eqref{m2} and \eqref{m3} in \eqref{m1}, we obtain
\begin{equation}\label{m4}
\|u_{\delta}\|^\frac{p(1-\delta-p)}{1-\delta}\leq \inf_{v\in S_\delta}\|v\|^{p}.
\end{equation}
Using \eqref{extremal} and \eqref{m4}, we obtain 
\begin{equation}\label{infprop}
\|U_\delta\|^p=\|u_{\delta}\|^\frac{p(1-\delta-p)}{1-\delta}\leq\inf_{v\in S_{\delta}}\|v\|^{p}.
\end{equation}
Since $U_{\delta}\in S_{\delta}$, from \eqref{infprop}, we obtain \eqref{bc}.

To prove the second part, let \eqref{SI} holds. If $S>\mu(\Omega)$, then from \eqref{bc} and \eqref{extremal} above, we obtain
\begin{equation}\label{con1}
S\Big(\int_{\Omega}U_\delta^{1-\delta}f\,dx\Big)^\frac{p}{1-\delta}>\int_{\Omega}|\nabla_{\textrm{H}} U_\delta|^p\,dx.
\end{equation}
Since $U_\delta\in W_0^{1,p}(\Omega)$, \eqref{con1} violates the hypothesis \eqref{SI}. Conversely, let
$$
S\leq \mu(\Omega)=\inf_{v\in S_{\delta}}\|v\|^{p}\leq \|U\|^{p},
$$ for all $U\in S_{\delta}$. We observe that the claim directly follows if $v=0$. So we consider the case of $v\in W_0^{1,p}(\Omega,w)\setminus\{0\}$. This gives
$$
U=\Bigg(\int_{\Omega}|v|^{1-\delta}f\,dx\Bigg)^{-\frac{1}{1-\delta}}v\in S_{\delta}.
$$
Therefore, we have
$$
S\leq\Bigg(\int_{\Omega}|v|^{1-\delta}f\,dx\Bigg)^{-\frac{p}{1-\delta}}\|v\|^p. 
$$
Hence, the result follows. \qed

\vskip 0.3cm

Department of Mathematics, Ben-Gurion University of the Negev, P.O.Box 653, Beer Sheva, 8410501, Israel

\emph{E-mail address:} \email{pgarain92@gmail.com}

\vskip 0.3cm

Department of Mathematics, Ben-Gurion University of the Negev, P.O.Box 653, Beer Sheva, 8410501, Israel

\emph{E-mail address:} \email{ukhlov@math.bgu.ac.il}

\end{document}